\documentclass[12pt]{article}%
\usepackage{amsmath}
\usepackage{amsfonts}
\usepackage{amssymb}
\usepackage{graphicx}%
\providecommand{\U}[1]{\protect\rule{.1in}{.1in}}
\newtheorem{theorem}{Theorem}
\newtheorem{acknowledgement}[theorem]{Acknowledgement}

\newtheorem{corollary}[theorem]{Corollary}

\newtheorem{example}[theorem]{Example}

\newtheorem{remark}[theorem]{Remark}

\newenvironment{proof}[1][Proof]{\noindent\textbf{#1.} }{\ \rule{0.5em}{0.5em}}
\begin{document}

\title{Fisher information of orthogonal polynomials I}
\author{Diego Dominici \thanks{e-mail: dominici@math.tu-berlin.de}\\Technische Universit\"{a}t Berlin\\Sekretariat MA 4-5\\Stra\ss e des 17. Juni 136 \\D-10623 Berlin \\Germany\\Permanent address: \\Department of Mathematics State University of New York at New Paltz \\1 Hawk Dr. \\New Paltz, NY 12561-2443 \\USA}
\date{\textit{Dedicated to Jes\'{u}s Dehesa on the occasion of his 60th birthday} }
\maketitle

MSC-class: 33C45, 94A17, 33C05

\begin{abstract}
Following the lead of J. Dehesa and his collaborators, we compute the Fisher
information of the Meixner-Pollaczek, Meixner, Krawtchouk and Charlier polynomials.

\end{abstract}

\section{Introduction}

The Fisher information $I_{\theta}\left(  \mu\right)  $ of a random variable
$X$ with distribution $\mu(x;\theta),$ where $\theta$ is a continuous
parameter, is defined by \cite{MR0346955}%
\begin{equation}
I_{\theta}\left(  \mu\right)  =\mathbb{E}\left\{  \left[  \frac{\partial
}{\partial\theta}\ln\left(  \mu\right)  \right]  ^{2}\right\}  .
\label{fisher}%
\end{equation}
It is named after R. A. Fisher (17 February 1890 -- 29 July 1962), who
invented the concept of maximum likelihood estimator and discovered many of
its properties. Among other results, he proved that if $\widehat{\theta}$ is
the maximum likelihood estimator of $\theta,$ we have the following asymptotic
normality of $\widehat{\theta}$
\[
\sqrt{n}\left(  \widehat{\theta}-\theta\right)  \overset{\mathcal{D}%
}{\rightarrow}\mathcal{N}\left(  0,\frac{1}{I_{\theta}\left(  \mu\right)
}\right)  ,
\]
where $\mathcal{N}(\cdot,\cdot)$ denotes the normal distribution and $n$ is
the sample size. Over the years,the concept of Fisher information has found
many applications in physics \cite{MR2069674}, biology \cite{MR2128009},
engineering \cite{MR1889650}, etc.

\begin{example}
The negative binomial distribution.

Let $0\leq p\leq1,$ $r>0$ and
\[
\mu(k;p,r)=\binom{r+k-1}{k}p^{k}(1-p)^{r},\quad k=0,1,\ldots.
\]
Then, we have%
\begin{equation}
I_{p}\left(  \mu\right)  =%
%TCIMACRO{\dsum \limits_{k=0}^{\infty}}%
%BeginExpansion
{\displaystyle\sum\limits_{k=0}^{\infty}}
%EndExpansion
\left[  \frac{r+k-r\left(  1-p\right)  }{p}\right]  ^{2}\binom{r+k-1}{k}%
p^{k}(1-p)^{r}=\frac{r}{p\left(  1-p\right)  ^{2}}. \label{negbin}%
\end{equation}

\end{example}

\begin{example}
The binomial distribution.

Let $0\leq p\leq1,$ $n\in\mathbb{N}$ and
\[
\mu(k;p,n)=\binom{n}{k}p^{k}(1-p)^{n-k},\quad k=0,1,\ldots,n.
\]
Then, we have%
\begin{equation}
I_{p}\left(  \mu\right)  =%
%TCIMACRO{\dsum \limits_{k=0}^{n}}%
%BeginExpansion
{\displaystyle\sum\limits_{k=0}^{n}}
%EndExpansion
\left[  \frac{k-pn}{p\left(  1-p\right)  }\right]  ^{2}\binom{n}{k}%
p^{k}(1-p)^{n-k}=\frac{n}{p\left(  1-p\right)  }. \label{bin}%
\end{equation}

\end{example}

\begin{example}
The Poisson distribution.

Let $\lambda>0$ and
\[
\mu(k;\lambda)=\frac{\lambda^{k}}{k!}e^{-\lambda},\quad k=0,1,\ldots.
\]
Then, we have%
\begin{equation}
I_{\lambda}\left(  \mu\right)  =%
%TCIMACRO{\dsum \limits_{k=0}^{\infty}}%
%BeginExpansion
{\displaystyle\sum\limits_{k=0}^{\infty}}
%EndExpansion
\left(  \frac{k-\lambda}{\lambda}\right)  ^{2}\frac{\lambda^{k}}%
{k!}e^{-\lambda}=\frac{1}{\lambda}. \label{poisson}%
\end{equation}

\end{example}

In \cite{MR2146490}, J. S\'{a}nchez-Ruiz and J. Dehesa introduced the concept
of Fisher information of orthogonal polynomials. They considered a sequence of
real polynomials orthogonal with respect to the weight function $\rho(x)$ on
the interval $[a,b]$%
\[%
%TCIMACRO{\dint \limits_{a}^{b}}%
%BeginExpansion
{\displaystyle\int\limits_{a}^{b}}
%EndExpansion
P_{n}(x)P_{m}(x)\rho(x)dx=h_{n}\delta_{n,m},\quad n,m=0,1,\ldots,
\]
with $\deg\left(  P_{n}\right)  =n.$ Introducing the normalized density
functions%
\begin{equation}
\rho_{n}\left(  x\right)  =\frac{\left[  P_{n}(x)\right]  ^{2}\rho\left(
x\right)  }{h_{n}}, \label{rhonjesus}%
\end{equation}
they defined the Fisher information corresponding to the densities
(\ref{rhonjesus}) by%
\begin{equation}
\mathcal{I}\left(  n\right)  =%
%TCIMACRO{\dint \limits_{a}^{b}}%
%BeginExpansion
{\displaystyle\int\limits_{a}^{b}}
%EndExpansion
\frac{\left[  \rho_{n}^{\prime}\left(  x\right)  \right]  ^{2}}{\rho
_{n}\left(  x\right)  }dx, \label{Ijesus}%
\end{equation}
which they referred to as the Fisher information of the polynomial $P_{n}(x).$
Applying (\ref{Ijesus}) to the classical hypergeometric polynomials, they
calculated $\mathcal{I}\left(  n\right)  $ for the Jacobi, Laguerre and
Hermite polynomials.

In this work, we extend their ideas to some families of orthogonal
polynomials. We use a concept of Fisher information closer to (\ref{fisher}),
i.e., information content with respect to a parameter.

The paper is organized as follows: Section \ref{s2} contains some general
results on hypergeometric polynomials which we have been unable to find
explicitly in the literature, although they may be known. In Section \ref{S3}
we compute the Fisher information of the Meixner-Pollaczek, Meixner,
Krawtchouk and Charlier polynomials.

\section{Preliminaries\label{s2}}

Let $P_{n}(x)$ be a family of orthogonal polynomials satisfying
\begin{equation}%
%TCIMACRO{\dsum \limits_{x=0}^{\infty}}%
%BeginExpansion
{\displaystyle\sum\limits_{x=0}^{\infty}}
%EndExpansion
P_{n}(x)P_{m}(x)\rho\left(  x\right)  =h_{n}\delta_{n,m},\quad n,m=0,1,\ldots.
\label{ortho}%
\end{equation}
We define%
\begin{equation}
\rho_{n}\left(  x\right)  =\frac{\left[  P_{n}(x)\right]  ^{2}\rho\left(
x\right)  }{h_{n}},\quad n=0,1,\ldots\label{rhon}%
\end{equation}
and%
\begin{equation}
I_{\theta}\left(  P_{n}\right)  =%
%TCIMACRO{\dsum \limits_{x=0}^{\infty}}%
%BeginExpansion
{\displaystyle\sum\limits_{x=0}^{\infty}}
%EndExpansion
\left[  \frac{\partial}{\partial\theta}\rho_{n}\left(  x\right)  \right]
^{2}\frac{1}{\rho_{n}\left(  x\right)  },\quad n=0,1,\ldots. \label{In}%
\end{equation}
Note that
\begin{equation}%
%TCIMACRO{\dsum \limits_{x=0}^{\infty}}%
%BeginExpansion
{\displaystyle\sum\limits_{x=0}^{\infty}}
%EndExpansion
\rho_{n}\left(  x\right)  =1,\quad n=0,1,\ldots. \label{rhon1}%
\end{equation}

\begin{theorem}
Let $P_{n}(x)$ be a family of polynomials defined by
\[
P_{n}(x)=~_{2}F_{1}\left[  \left.
\begin{array}
[c]{c}%
-n,-x\\
c
\end{array}
\right\vert z\left(  \theta\right)  \right]  ,\quad n=0,1,\ldots,
\]
where $_{2}F_{1}\left[  \cdot\right]  $ is the hypergeometric function
\cite{MR0350075}%
\[
~_{2}F_{1}\left[  \left.
\begin{array}
[c]{c}%
a,b\\
c
\end{array}
\right\vert z\right]  =%
%TCIMACRO{\dsum \limits_{k=0}^{\infty}}%
%BeginExpansion
{\displaystyle\sum\limits_{k=0}^{\infty}}
%EndExpansion
\frac{\left(  a\right)  _{k}\left(  b\right)  _{k}}{\left(  c\right)  _{k}%
}\frac{z^{k}}{k!}%
\]
and $\left(  \cdot\right)  _{k}$ denotes the Pochhammer symbol. Then,
$P_{n}(x)$ satisfies the following:

\begin{enumerate}
\item
\begin{equation}
\left(  n+c\right)  P_{n+1}\left(  x\right)  +\left[  \left(  n-x\right)
z-2n-c\right]  P_{n}\left(  x\right)  -n\left(  z-1\right)  P_{n-1}\left(
x\right)  =0. \label{recu}%
\end{equation}

\item
\begin{equation}
\frac{\partial P_{n}}{\partial\theta}=n\frac{z^{\prime}}{z}\left[
P_{n}\left(  x\right)  -P_{n-1}\left(  x\right)  \right]  ,\quad n=0,1,\ldots.
\label{diff}%
\end{equation}

\item
\[%
%TCIMACRO{\dsum \limits_{x=0}^{\infty}}%
%BeginExpansion
{\displaystyle\sum\limits_{x=0}^{\infty}}
%EndExpansion
P_{n}(x)P_{m}(x)\rho\left(  x\right)  =h_{n}\delta_{n,m},\quad n,m=0,1,\ldots
,
\]
with%
\begin{equation}
\rho\left(  x\right)  =\frac{\left(  c\right)  _{x}}{\left(  1-z\right)
^{x}x!} \label{rho}%
\end{equation}
and%
\begin{equation}
h_{n}=\left(  1-\frac{1}{z}\right)  ^{c}\frac{\left(  1-z\right)  ^{n}%
}{\left(  c\right)  _{n}}n!,\quad n=0,1,\ldots. \label{Kn}%
\end{equation}

\end{enumerate}
\end{theorem}

\begin{proof}
The three term recurrence equation (\ref{recu}) is a direct consequence of the
contiguous relation \cite[2.8 (31)]{MR698779}%
\[
\left[  2-2a-\left(  b-a\right)  z\right]  F+a(1-z)F(a+1)-(c-a)F(a-1)=0,
\]
where
\[
F=~_{2}F_{1}\left[  \left.
\begin{array}
[c]{c}%
a,b\\
c
\end{array}
\right\vert z\right]  .
\]
The differentiation formula (\ref{diff}) follows from the identity
\cite[(2.5.5)]{MR1688958}.%
\[
z\frac{dF}{dz}=a\left[  F(a+1)-F(a)\right]  .
\]

To prove (\ref{rho}),(\ref{Kn}), we use the formula \cite[2.5.2 (12)]%
{MR698779}%
\begin{align*}
&
%TCIMACRO{\dsum \limits_{k=0}^{\infty}}%
%BeginExpansion
{\displaystyle\sum\limits_{k=0}^{\infty}}
%EndExpansion
\binom{\lambda}{k}s^{k}~_{2}F_{1}\left[  \left.
\begin{array}
[c]{c}%
-k,b\\
-\lambda
\end{array}
\right\vert z\right]  ~_{2}F_{1}\left[  \left.
\begin{array}
[c]{c}%
-k,\beta\\
-\lambda
\end{array}
\right\vert \zeta\right] \\
&  =\frac{(1+s)^{\lambda+b+\beta}}{\left[  1+s(1-z)\right]  ^{b}\left[
1+s(1-\zeta)\right]  ^{\beta}}\ _{2}F_{1}\left[  \left.
\begin{array}
[c]{c}%
b,\beta\\
-\lambda
\end{array}
\right\vert \frac{-z\zeta s}{\left[  1+s(1-z)\right]  \left[  1+s(1-\zeta
)\right]  }\right]  ,
\end{align*}
with $\lambda=-c,$ $b=-n,$ $\beta=-m$ and $s=\left(  z-1\right)  ^{-1}.$
Taking into account that%
\[
\rho\left(  x\right)  =\frac{\left(  c\right)  _{x}}{\left(  1-z\right)
^{x}x!}=\binom{-c}{x}\left(  z-1\right)  ^{-x},
\]
we have%
\begin{align*}
&
%TCIMACRO{\dsum \limits_{x=0}^{\infty}}%
%BeginExpansion
{\displaystyle\sum\limits_{x=0}^{\infty}}
%EndExpansion
\ _{2}F_{1}\left[  \left.
\begin{array}
[c]{c}%
-n,-x\\
c
\end{array}
\right\vert z\right]  ~_{2}F_{1}\left[  \left.
\begin{array}
[c]{c}%
-m,-x\\
c
\end{array}
\right\vert \zeta\right]  \rho\left(  x\right) \\
&  =\left(  1-\frac{1}{z}\right)  ^{c}\ \left(  \frac{z-\zeta}{z}\right)
^{n+m}\ _{2}F_{1}\left[  \left.
\begin{array}
[c]{c}%
-n,-m\\
c
\end{array}
\right\vert \frac{(1-z)\zeta^{2}}{\left(  z-\zeta\right)  ^{2}}\right]  .
\end{align*}
Assuming that $n\leq m,$ we get%
\begin{align}
&
%TCIMACRO{\dsum \limits_{x=0}^{\infty}}%
%BeginExpansion
{\displaystyle\sum\limits_{x=0}^{\infty}}
%EndExpansion
\ _{2}F_{1}\left[  \left.
\begin{array}
[c]{c}%
-n,-x\\
c
\end{array}
\right\vert z\right]  ~_{2}F_{1}\left[  \left.
\begin{array}
[c]{c}%
-m,-x\\
c
\end{array}
\right\vert \zeta\right]  \rho\left(  x\right) \label{hyper1}\\
&  =\left(  \frac{z-1}{z}\right)  ^{c}\ \left(  \frac{1}{z}\right)  ^{n+m}%
%TCIMACRO{\dsum \limits_{k=0}^{n}}%
%BeginExpansion
{\displaystyle\sum\limits_{k=0}^{n}}
%EndExpansion
\frac{\left(  -n\right)  _{k}\left(  -m\right)  _{k}}{\left(  c\right)  _{k}%
}\frac{\left(  1-z\right)  ^{k}\zeta^{2k}\left(  z-\zeta\right)  ^{n+m-2k}%
}{k!}.\nonumber
\end{align}
Replacing $\zeta$ by $z$ in (\ref{hyper1}), we obtain%
\begin{align*}
&
%TCIMACRO{\dsum \limits_{x=0}^{\infty}}%
%BeginExpansion
{\displaystyle\sum\limits_{x=0}^{\infty}}
%EndExpansion
\ _{2}F_{1}\left[  \left.
\begin{array}
[c]{c}%
-n,-x\\
c
\end{array}
\right\vert z\right]  ~_{2}F_{1}\left[  \left.
\begin{array}
[c]{c}%
-m,-x\\
c
\end{array}
\right\vert z\right]  \rho\left(  x\right) \\
&  =\left(  1-\frac{1}{z}\right)  ^{c}\ \frac{\left[  \left(  -n\right)
_{n}\right]  ^{2}}{\left(  c\right)  _{n}}\frac{\left(  1-z\right)  ^{n}}%
{n!}\delta_{n,m}%
\end{align*}
and the result follows.
\end{proof}

\begin{corollary}
Let
\[
P_{n}(x)=~_{2}F_{0}\left[  \left.
\begin{array}
[c]{c}%
-n,-x\\
-
\end{array}
\right\vert z\left(  \theta\right)  \right]  ,\quad n=0,1,\ldots.
\]

Then,

\begin{enumerate}
\item
\begin{equation}
P_{n+1}\left(  x\right)  +\left[  \left(  x-n\right)  z-1\right]  P_{n}\left(
x\right)  -nzP_{n-1}\left(  x\right)  =0. \label{recu0}%
\end{equation}

\item
\begin{equation}
\frac{\partial P_{n}}{\partial\theta}=n\frac{z^{\prime}}{z}\left[
P_{n}\left(  x\right)  -P_{n-1}\left(  x\right)  \right]  ,\quad n=0,1,\ldots.
\label{diff0}%
\end{equation}

\item We have%
\begin{equation}
\rho\left(  x\right)  =\frac{\left[  -z\left(  \theta\right)  \right]  ^{-x}%
}{x!} \label{rho0}%
\end{equation}
and%
\begin{equation}
h_{n}=\left[  -z\left(  \theta\right)  \right]  ^{n}n!\exp\left[  -\frac
{1}{z\left(  \theta\right)  }\right]  ,\quad n=0,1,\ldots, \label{Kn0}%
\end{equation}
where $\rho\left(  x\right)  $ and $h_{n}$ were defined in (\ref{ortho}).
\end{enumerate}
\end{corollary}

\begin{proof}
The results follow immediately using the limit relation \cite[(0.4.5)]%
{koekoek94askeyscheme}%

\begin{equation}
~_{2}F_{0}\left[  \left.
\begin{array}
[c]{c}%
-n,-x\\
-
\end{array}
\right\vert z\left(  \theta\right)  \right]  =~\underset{\lambda
\rightarrow\infty}{\lim}\ _{2}F_{1}\left[  \left.
\begin{array}
[c]{c}%
-n,-x\\
\lambda c
\end{array}
\right\vert \lambda cz\left(  \theta\right)  \right]  . \label{limit}%
\end{equation}

\end{proof}

\section{Main results \label{S3}}

We shall now use that results of the previous section and compute the Fisher
information of the Meixner-Pollaczek, Meixner, Krawtchouk and Charlier polynomials.

\begin{theorem}
\label{Th1}Let
\[
P_{n}(x)=~_{2}F_{1}\left[  \left.
\begin{array}
[c]{c}%
-n,-x\\
c
\end{array}
\right\vert z\left(  \theta\right)  \right]  ,\quad n=0,1,\ldots.
\]
Then,%
\begin{equation}
I_{\theta}\left(  P_{n}\right)  =\left(  \frac{z^{\prime}}{z}\right)
^{2}\left(  1-z\right)  ^{-1}\left[  2n^{2}+\left(  2n+1\right)  c\right]
,\quad n=0,1,\ldots.\label{In1}%
\end{equation}

\end{theorem}

\begin{proof}
From (\ref{rhon}), (\ref{rho}) and (\ref{Kn}), we have%
\begin{equation}
\rho_{n}(x)=\left(  1-\frac{1}{z}\right)  ^{-c}\left[  P_{n}(x)\right]
^{2}\frac{\left(  c\right)  _{x}\left(  c\right)  _{n}}{\left(  1-z\right)
^{x+n}x!n!}. \label{rhonnew}%
\end{equation}
Hence,%
\begin{equation}
\frac{\partial}{\partial\theta}\rho_{n}\left(  x\right)  =-\left(  \frac
{z}{z-1}\right)  ^{c-1}\frac{\left(  c\right)  _{x}\left(  c\right)  _{n}%
}{\left(  1-z\right)  ^{x+n}x!n!}P_{n}(x)\left[  \left(  c+nz+xz\right)
z^{\prime}P_{n}(x)+2z(1-z)\frac{\partial P_{n}}{\partial\theta}\right]  .
\label{rho1}%
\end{equation}
Using (\ref{diff}) and (\ref{recu}) in (\ref{rho1}), we obtain%
\begin{align}
\frac{\partial}{\partial\theta}\rho_{n}\left(  x\right)   &  =-\left(
\frac{z}{z-1}\right)  ^{c-1}\frac{\left(  c\right)  _{x}\left(  c\right)
_{n}}{\left(  1-z\right)  ^{x+n}x!n!}z^{\prime}P_{n}(x)\left[  \left(
c-nz+xz+2n\right)  P_{n}(x)+2n(z-1)P_{n-1}\left(  x\right)  \right]
\label{rho2}\\
&  =-\left(  \frac{z}{z-1}\right)  ^{c-1}\frac{\left(  c\right)  _{x}\left(
c\right)  _{n}}{\left(  1-z\right)  ^{x+n}x!n!}z^{\prime}P_{n}(x)\left[
\left(  n+c\right)  P_{n+1}(x)+n(z-1)P_{n-1}\left(  x\right)  \right]
.\nonumber
\end{align}
Therefore,%
\begin{align}
\left[  \frac{\partial}{\partial\theta}\rho_{n}\left(  x\right)  \right]
^{2}\frac{1}{\rho_{n}\left(  x\right)  }  &  =\left(  \frac{z}{z-1}\right)
^{c-2}\frac{\left(  c\right)  _{x}\left(  c\right)  _{n}}{\left(  1-z\right)
^{x+n}x!n!}\left(  z^{\prime}\right)  ^{2}\left[  \left(  n+c\right)
P_{n+1}(x)+n(z-1)P_{n-1}\left(  x\right)  \right]  ^{2}\nonumber\\
&  =\left(  1-z\right)  ^{-1}\left(  \frac{z^{\prime}}{z}\right)  ^{2}\left[
\left(  n+1\right)  \left(  c+n\right)  \rho_{n+1}(x)+n\left(  c+n-1\right)
\rho_{n-1}(x)\right] \label{rho3}\\
&  +2\left(  \frac{z}{z-1}\right)  ^{c-2}\frac{\left(  c\right)  _{n}}{\left(
1-z\right)  ^{n}n!}\left(  z^{\prime}\right)  ^{2}n\left(  n+c\right)
(z-1)P_{n+1}(x)P_{n-1}\left(  x\right)  \rho\left(  x\right)  ,\nonumber
\end{align}
where we have used (\ref{rho}) and (\ref{rhonnew}).

Summing (\ref{rho3}), while taking (\ref{ortho}), (\ref{In}) and (\ref{rhon1})
into account, the result follows.
\end{proof}

\begin{corollary}
\label{coro1}Let
\[
P_{n}(x)=~_{2}F_{0}\left[  \left.
\begin{array}
[c]{c}%
-n,-x\\
-
\end{array}
\right\vert z\left(  \theta\right)  \right]  ,\quad n=0,1,\ldots.
\]
Then,%
\[
I_{\theta}\left(  P_{n}\right)  =-\left(  \frac{z^{\prime}}{z}\right)
^{2}\frac{\left(  2n+1\right)  }{z},\quad n=0,1,\ldots.
\]

\end{corollary}

\begin{proof}
The result follows from (\ref{recu0})-(\ref{Kn0}) and the same steps used in
the proof of Theorem \ref{Th1}. It can also be proved directly by using
(\ref{limit}) in (\ref{In1}), since%
\[
\underset{\lambda\rightarrow\infty}{\lim}\left(  \frac{\lambda cz^{\prime}%
}{\lambda cz}\right)  ^{2}\left(  1-\lambda cz\right)  ^{-1}\left[
2n^{2}+\left(  2n+1\right)  \lambda c\right]  =-\left(  \frac{z^{\prime}}%
{z}\right)  ^{2}\frac{\left(  2n+1\right)  }{z}.
\]
We have now all the elements to state our main result.
\end{proof}

\begin{theorem}
The Fisher information of the Meixner, Krawtchouk and Charlier polynomials is
given by:

\begin{enumerate}
\item Meixner%
\[
I_{c}\left(  M_{n}\right)  =\frac{2n^{2}+(2n+1)\beta}{c\left(  c-1\right)
^{2}},\quad n=0,1,\ldots.
\]

\item Krawtchouk%
\[
I_{p}\left(  K_{n}\right)  =\frac{2n^{2}-(2n+1)N}{p\left(  p-1\right)  },\quad
n=0,1,\ldots,N.
\]

\item Charlier%
\[
I_{a}\left(  C_{n}\right)  =\frac{2n+1}{a},\quad n=0,1,\ldots.
\]

\end{enumerate}
\end{theorem}

\begin{proof}
The result follows from Theorem \ref{Th1}, Corollary \ref{coro1} and the
hypergeometric representations \cite{koekoek94askeyscheme}%
\begin{align*}
M_{n}(x;\beta,c)  &  =~_{2}F_{1}\left[  \left.
\begin{array}
[c]{c}%
-n,-x\\
\beta
\end{array}
\right\vert 1-\frac{1}{c}\right]  ,\quad\beta>0,\quad0<c<1,\\
K_{n}(x;p,N)  &  =~_{2}F_{1}\left[  \left.
\begin{array}
[c]{c}%
-n,-x\\
-N
\end{array}
\right\vert \frac{1}{p}\right]  ,\quad0<p<1,\quad N=0,1,\ldots,\\
C_{n}(x;a)  &  =~_{2}F_{0}\left[  \left.
\begin{array}
[c]{c}%
-n,-x\\
-
\end{array}
\right\vert -\frac{1}{a}\right]  ,\quad a>0.
\end{align*}

\end{proof}

\begin{remark}
When $n=0,$ we recover the Fisher information of the negative binomial
(\ref{negbin}), binomial (\ref{bin}) and Poisson (\ref{poisson}) distributions.
\end{remark}

We now compute the Fisher information of the Meixner-Pollaczek polynomials
using some of the results obtained in the discrete case.

\begin{theorem}
The Fisher information of the Meixner-Pollaczek polynomials is given by:%
\[
I_{\phi}\left(  P_{n}^{\left(  \lambda\right)  }\right)  =%
%TCIMACRO{\dint \limits_{-\infty}^{\infty}}%
%BeginExpansion
{\displaystyle\int\limits_{-\infty}^{\infty}}
%EndExpansion
\left[  \frac{\partial}{\partial\theta}\rho_{n}\left(  x\right)  \right]
^{2}\frac{1}{\rho_{n}\left(  x\right)  }dx=\frac{2\left[  n^{2}+\left(
2n+1\right)  \lambda\right]  }{\sin^{2}\left(  \phi\right)  },\quad
n=0,1,\ldots,
\]
with $\rho_{n}\left(  x\right)  $ defined as in (\ref{rhon}).
\end{theorem}

\begin{proof}
The Meixner-Pollaczek polynomials have the hypergeometric representation
\cite{koekoek94askeyscheme}%
\begin{equation}
P_{n}^{\left(  \lambda\right)  }(x;\phi)=\frac{\left(  2\lambda\right)  _{n}%
}{n!}e^{\mathrm{i}n\phi}\ ~_{2}F_{1}\left[  \left.
\begin{array}
[c]{c}%
-n,\lambda+\mathrm{i}x\\
2\lambda
\end{array}
\right\vert 1-e^{-2\mathrm{i}\phi}\right]  ,\quad\lambda>0,\quad0<\phi
<\pi.\label{MP1}%
\end{equation}
They satisfy the orthogonality relation%
\begin{equation}
\frac{1}{2\pi}%
%TCIMACRO{\dint \limits_{-\infty}^{\infty}}%
%BeginExpansion
{\displaystyle\int\limits_{-\infty}^{\infty}}
%EndExpansion
e^{\left(  2\phi-\pi\right)  x}\left\vert \Gamma\left(  \lambda+\mathrm{i}%
x\right)  \right\vert ^{2}P_{m}^{\left(  \lambda\right)  }(x;\phi
)P_{n}^{\left(  \lambda\right)  }(x;\phi)dx=\frac{\Gamma\left(  n+2\lambda
\right)  }{\left[  2\sin\left(  \phi\right)  \right]  ^{2\lambda}n!}%
\delta_{n,m},\quad n,m=0,1,\ldots\label{MP2}%
\end{equation}
and the recurrence relation%
\begin{equation}
\left(  n+1\right)  P_{n+1}^{\left(  \lambda\right)  }-2\left[  x\sin\left(
\phi\right)  +\left(  n+\lambda\right)  \cos\left(  \phi\right)  \right]
P_{n}^{\left(  \lambda\right)  }+\left(  n+2\lambda-1\right)  P_{n-1}^{\left(
\lambda\right)  }=0.\label{MPrecu}%
\end{equation}

From (\ref{diff}) and (\ref{MP1}), we have%
\[
\frac{\partial P_{n}^{\left(  \lambda\right)  }}{\partial\phi}=n\cot\left(
\phi\right)  P_{n}^{\left(  \lambda\right)  }-\frac{\left(  n+2\lambda
-1\right)  }{\sin\left(  \phi\right)  }P_{n-1}^{\left(  \lambda\right)  },
\]
while (\ref{rhon}) and (\ref{MP2}) give%
\begin{equation}
\rho_{n}(x)=\frac{e^{\left(  2\phi-\pi\right)  x}\left\vert \Gamma\left(
\lambda+\mathrm{i}x\right)  \right\vert ^{2}\left[  2\sin\left(  \phi\right)
\right]  ^{2\lambda}n!\left[  P_{n}^{\left(  \lambda\right)  }(x;\phi)\right]
^{2}}{2\pi\Gamma\left(  n+2\lambda\right)  }.\label{MP3}%
\end{equation}
Note that%
\begin{equation}%
%TCIMACRO{\dint \limits_{-\infty}^{\infty}}%
%BeginExpansion
{\displaystyle\int\limits_{-\infty}^{\infty}}
%EndExpansion
\rho_{n}(x)dx=1,\quad n=0,1,\ldots.\label{one}%
\end{equation}
Differentiating (\ref{MP3}) with respect to $\phi,$ we obtain%
\[
\frac{\partial\rho_{n}}{\partial\phi}=\frac{2\rho_{n}(x)}{P_{n}^{\left(
\lambda\right)  }}\left\{  \left[  x+\left(  n+\lambda\right)  \cot\left(
\phi\right)  \right]  P_{n}^{\left(  \lambda\right)  }-\frac{\left(
n+2\lambda-1\right)  }{\sin\left(  \phi\right)  }P_{n-1}^{\left(
\lambda\right)  }\right\}
\]
or, using (\ref{MPrecu}),%
\begin{equation}
\frac{\partial\rho_{n}}{\partial\phi}=\frac{\rho_{n}(x)}{\sin\left(
\phi\right)  P_{n}^{\left(  \lambda\right)  }}\left[  \left(  n+1\right)
P_{n+1}^{\left(  \lambda\right)  }-\left(  n+2\lambda-1\right)  P_{n-1}%
^{\left(  \lambda\right)  }\right]  .\label{Mp4}%
\end{equation}
Therefore,%
\begin{gather}
\left[  \frac{\partial}{\partial\theta}\rho_{n}\left(  x\right)  \right]
^{2}\frac{1}{\rho_{n}\left(  x\right)  }=\frac{\rho_{n}\left(  x\right)
}{\sin^{2}\left(  \phi\right)  \left(  P_{n}^{\left(  \lambda\right)
}\right)  ^{2}}\left[  \left(  n+1\right)  ^{2}\left(  P_{n+1}^{\left(
\lambda\right)  }\right)  ^{2}\right.  \nonumber\\
\left.  -2\left(  n+1\right)  \left(  n+2\lambda-1\right)  P_{n+1}^{\left(
\lambda\right)  }P_{n-1}^{\left(  \lambda\right)  }+\left(  n+2\lambda
-1\right)  ^{2}\left(  P_{n-1}^{\left(  \lambda\right)  }\right)  ^{2}\right]
\label{Mp5}\\
=\frac{1}{\sin^{2}\left(  \phi\right)  }\left[  \left(  n+1\right)  \left(
n+2\lambda\right)  \rho_{n+1}\left(  x\right)  +n\left(  n+2\lambda-1\right)
\rho_{n-1}\left(  x\right)  \right.  \nonumber\\
\left.  -2\left(  n+1\right)  \left(  n+2\lambda-1\right)  \frac{\left[
2\sin\left(  \phi\right)  \right]  ^{2\lambda}n!}{\Gamma\left(  n+2\lambda
\right)  }\rho(x)P_{n+1}^{\left(  \lambda\right)  }P_{n-1}^{\left(
\lambda\right)  }\right]  ,\nonumber
\end{gather}
where%
\[
\rho(x)=\frac{e^{\left(  2\phi-\pi\right)  x}\left\vert \Gamma\left(
\lambda+\mathrm{i}x\right)  \right\vert ^{2}}{2\pi}.
\]

Integrating (\ref{Mp5}) and using the orthogonality relation (\ref{MP2}) and
(\ref{one}), we get%
\[%
%TCIMACRO{\dint \limits_{-\infty}^{\infty}}%
%BeginExpansion
{\displaystyle\int\limits_{-\infty}^{\infty}}
%EndExpansion
\left[  \frac{\partial}{\partial\theta}\rho_{n}\left(  x\right)  \right]
^{2}\frac{1}{\rho_{n}\left(  x\right)  }dx=\frac{1}{\sin^{2}\left(
\phi\right)  }\left[  \left(  n+1\right)  \left(  n+2\lambda\right)  +n\left(
n+2\lambda-1\right)  \right]
\]
and the result follows.
\end{proof}

\section{Conclusions and further directions}

We have computed the Fisher information of the Meixner, Krawtchouk and
Charlier polynomials, which can be viewed in a sense as extensions of the
Fisher information of the negative binomial, binomial and Poisson
distributions, respectively. We are working on trying to extend the same
framework to include other discrete orthogonal polynomials, namely the Racah
and Hahn families. 

We have also obtained the Fisher information of the Meixner-Pollaczek
polynomials. It would be very interesting to calculate the Fisher information
of the Wilson and the rest of the Hahn families (continuous, dual and
continuous dual). Finally, the Fisher information of $q$-orthogonal
polynomials doesn't seem to have been considered yet.

\begin{acknowledgement}
This work was conducted while visiting Technische Universit\"{a}t Berlin and
supported in part by a Sofja Kovalevskaja Award from the Humboldt Foundation,
provided by Professor Olga Holtz. We wish to thank Olga for her generous
sponsorship and our colleagues at TU Berlin for their continuous help.
\end{acknowledgement}

\bibliographystyle{abbrv}

\end{document}